\documentclass[12pt]{article}
\usepackage{amsfonts}
\usepackage{amssymb}
\usepackage{amsmath}
\usepackage{theorem}

\usepackage{tikz}

\usepackage[normalem]{ulem}

\usepackage{hyperref}

\usepackage[all]{xy}

\usetikzlibrary{matrix}
\usetikzlibrary{arrows}

\numberwithin{equation}{section}

\input amssym.def

\makeatletter
\newcommand\qedsymbol{\hbox{$\Box$}}
\newcommand\qed{\relax\ifmmode\Box\else
  {\unskip\nobreak\hfil\penalty50\hskip1em\null\nobreak\hfil\qedsymbol
  \parfillskip=\z@\finalhyphendemerits=0\endgraf}\fi}

\newenvironment{proof}[1][{}]{\par\noindent Proof{#1}. }{\qed}

\newcommand{\bfzero}{{\mathbf 0}}

\newcommand{\End}{{\mathrm{End}}}

\newcommand{\Hom}{{\mathrm{Hom}}}
\newcommand{\Objects}{{\mathrm{Objects}}}

\newcommand{\Sh}{{\mathrm{Sh}}}

\newcommand{\Cobar}{{\mathrm{Cobar}}}

\newcommand{\sgn}{{\rm s g n}}

\newcommand{\hotimes}{{\,\hat{\otimes}\,}}

\newcommand{\Lie}{{\mathsf{Lie}}}

\newcommand{\Com}{{\mathsf{Com}}}
\newcommand{\coCom}{{\mathsf{coCom}}}

\newcommand{\Ch}{{\mathsf{Ch}}}

\newcommand{\bul}{{\bullet}}
\newcommand{\al}{{\alpha}}

\newcommand{\mC}{{\mathfrak{C}}}

\newcommand{\mS}{{\mathfrak{S}}}

\newcommand{\bs}{{\bf s}}
\newcommand{\bsi}{{\bf s}^{-1}\,}

\newcommand{\Om}{{\Omega}}

\newcommand{\si}{{\sigma}}

\newcommand{\ve}{{\varepsilon}}

\newcommand{\pa}{{\partial}}

\newcommand{\cF}{{\cal F}}

\newcommand{\cL}{{\cal L}}
\newcommand{\cC}{{\mathcal{C}}}

\newcommand{\bbk}{{\Bbbk}}

\newcommand{\bbZ}{{\mathbb Z}}

\newcommand{\D}{{\Delta}}

\newcommand{\wt}[1]{{\widetilde{#1}}}
\newcommand{\ti}[1]{{\tilde{#1}}}
\newcommand{\wh}[1]{{\widehat{#1}}}

\newcommand{\Map}{{\mathbf{map}}}

\newcommand{\maps}{\,\colon\,}

\newcommand{\und}[1]{{\underline{#1}}}

\newcommand{\sLie}{{\mathfrak{S}}\mathsf{Lie}_{\infty}}
\newcommand{\tLie}{{\mathfrak{S}}\mathsf{Lie}_{\infty}^{\mathrm{MC}}}

\newcommand{\tensor}{\otimes}
\newcommand{\xto}[1]{\xrightarrow{#1}}

\newcommand{\MC}{\mathrm{MC}}
\newcommand{\mMC}{\mathfrak{MC}}
\newcommand{\mSH}{{\mathfrak{SH}}}
\newcommand{\curv}{{\mathsf{curv}}}
\newcommand{\Shift}{{\mathrm{Shift}}}

\newtheorem{thm}{Theorem}[section]
\newtheorem{defi}[thm]{Definition}

\newtheorem{lem}[thm]{Lemma}

\newtheorem{prop}[thm]{Proposition}
\newtheorem{claim}[thm]{Claim}

\theorembodyfont{\rm}

\newtheorem{remark}[thm]{Remark}

\title{On an enhancement of the category of shifted $L_{\infty}$-algebras}

\author{Vasily A. Dolgushev and Christopher L. Rogers}

\date{}

\begin{document}

\maketitle

\begin{abstract}
  We construct a symmetric monoidal category $\tLie$ whose objects are
  shifted $L_{\infty}$-algebras equipped with a complete descending
  filtration. Morphisms of this category are ``enhanced'' infinity
  morphisms between shifted $L_{\infty}$-algebras.  We prove that any
  category enriched over $\tLie$ can be integrated to a
  simplicial category whose mapping spaces are Kan complexes.
  The advantage gained by using enhanced morphisms is that we can see much more of the simplicial world from the
    $L_{\infty}$-algebra point of view.  We use this construction in a subsequent paper
  \cite{HAform} to produce a simplicial model of a
    $(\infty,1)$-category whose objects are homotopy algebras of a
  fixed type.
\end{abstract}

\section{Introduction}
Let $L$ be a $\bbZ$-graded $\bbk$-vector space\footnote{In this paper, 
we assume that $\textrm{char}(\bbk) = 0$.}.
A shifted $L_{\infty}$-algebra structure on $L$ is a degree 
$1$ coderivation $Q$ of the cocommutative coalgebra 
$$
\und{S}(L) : = L \oplus S^2(L) \oplus S^3(L) \oplus \dots
$$
satisfying the Maurer-Cartan (MC) equation 
$$
[Q, Q] = 0\,.
$$
An $\infty$-morphism $F: (L,Q)  \to  (\wt{L}, \wt{Q})$ is a homomorphism of 
dg cocommutative coalgebras 
\begin{equation}
\label{infty-morph}
F : (\und{S}(L), Q) \to  (\und{S}(\wt{L}), \wt{Q})\,.
\end{equation}
It is often convenient to extend any such coalgebra homomorphism 
\eqref{infty-morph} to the homomorphism of coalgebras $S(L)$, $S(\wt{L})$ with counits 
by requiring that 
$$
F(1) = 1\,.
$$  

It is known that the category of shifted $L_{\infty}$-algebras
is naturally a symmetric monoidal category: given two shifted $L_{\infty}$-algebras 
$L_1$, $L_2$ their ``tensor product'' is defined as the direct sum 
$L_1 \oplus L_2$; given two $\infty$-morphisms 
$$
F_1 : \und{S}(L_1) \to \und{S}(\wt{L_1})\,, 
\qquad 
F_2 : \und{S}(L_2) \to \und{S}(\wt{L_2})\,, 
$$  
their tensor product\footnote{Even though the ``tensor product'' of (shifted) $L_{\infty}$-algebras 
is $\oplus$, it is more convenient to use the notation $\otimes$ for the tensor product 
of $\infty$-morphisms of (shifted) $L_{\infty}$-algebras.}
is defined as 
$$
F_1 \otimes F_2 : =  F_1 \otimes F_2 \Big|_{\und{S}(L_1 \oplus L_2)} : \und{S}(L_1 \oplus L_2) 
\to \und{S}(\wt{L}_1 \oplus \wt{L}_2)\,.
$$
It is clear that $\bfzero$ is the unit object of this symmetric monoidal category.

In this paper, we construct a useful enhancement $\tLie$ of the symmetric 
monoidal category of shifted $L_{\infty}$-algebras. Objects of the category $\tLie$
are shifted $L_{\infty}$-algebras equipped with a complete descending filtration.
A morphism from an object $L$ to an object $\wt{L}$ of $\tLie$ is a pair 
$$
(\al, F)
$$
where $\al$ is a MC element of $\wt{L}$, $F$ is a continuous 
$\infty$-morphism from $L$ to $\wt{L}^{\al}$, and the shifted $L_{\infty}$-algebra
$\wt{L}^{\al}$ is obtained from $\wt{L}$ via twisting by $\al$.

We use the Getzler-Hinich construction \cite{Ezra-infty},  \cite{Hinich:1997} to show that 
every $\tLie$-enriched category can be ``integrated'' to a simplicial category with 
mapping spaces being Kan complexes. This result illuminates
the main reason for introducing enhanced morphisms. Indeed, 
to an ordinary $\infty$-morphism $F \maps L_1
 \to L_2$, the Getzler-Hinich construction assigns a map between simplicial sets which must preserve the canonical
 base point corresponding to the trivial MC element $0$. This is too
 restrictive, since it prevents us from modeling more general simplicial
 maps. By using enhanced morphisms, we see
more of the simplicial hom-set at the level of
 $\sLie$-algebras. We can, for example, model simplicial maps
 from the point $\Delta^{0}$ as $\tLie$ morphisms originating from the unit object $\bfzero$.

In a subsequent paper \cite{HAform}, we show that homotopy algebras of 
a fixed type form a $\tLie$-enriched category and use this fact 
to produce a simplicial enrichment of the category 
of homotopy algebras. Furthermore, we prove that this
  simplicial category is a model for the $(\infty,1)$-category of  homotopy algebras.

~\\

\noindent
\textbf{Acknowledgements:} We would like to thank Thomas Willwacher 
for useful discussions, and Jim Stasheff for helpful comments.  V.A.D. and C.L.R. acknowledge NSF grant DMS-1161867.
C.L.R. also acknowledges support from the German Research Foundation 
(Deutsche Forschungsgemeinschaft (DFG)) through the Institutional Strategy of the
University of G\"ottingen.  V.A.D. also acknowledges a partial support from  
NSF grant DMS-1501001.

\subsubsection*{Notation and conventions} 
The ground field $\bbk$ has characteristic zero.
For most of the algebraic structures considered 
here, the underlying symmetric monoidal category is 
the category $\Ch_{\bbk}$ of unbounded cochain complexes 
of $\bbk$-vector spaces.  We will frequently use the ubiquitous
combination ``dg'' (differential graded) to refer to algebraic objects 
in $\Ch_{\bbk}$\,. 
For a cochain complex $V$, we denote 
by $\bs V$ (resp. by $\bs^{-1} V$) the suspension (resp. the 
desuspension) of $V$\,. In other words, 
$$
\big(\bs V\big)^{\bul} = V^{\bul-1}\,,  \qquad
\big(\bs^{-1} V\big)^{\bul} = V^{\bul+1}\,. 
$$

Any $\bbZ$-graded vector space $V$ is tacitly considered 
as the cochain complex with the zero differential. 

For a pair $V$, $W$ of $\bbZ$-graded vector spaces, we denote by 
$$
\Hom (V,W)
$$
the corresponding inner-hom object in the category of $\bbZ$-graded vector spaces, i.e. 
\begin{equation}
\label{Hom-gr} 
\Hom (V,W) : = \bigoplus_{m} \Hom^m_{\bbk}(V, W)\,,
\end{equation}
where $\Hom^m_{\bbk}(V, W)$ consists of $\bbk$-linear maps $f : V \to W$ such that 
$$
f(V^{\bul}) \subset W^{\bul+m}\,.
$$

The notation $S_{n}$ is reserved for the symmetric group 
on $n$ letters and  $\Sh_{p_1, \dots, p_k}$ denotes 
the subset of $(p_1, \dots, p_k)$-shuffles 
in $S_n$, i.e.  $\Sh_{p_1, \dots, p_k}$ consists of 
elements $\si \in S_n$, $n= p_1 +p_2 + \dots + p_k$ such that 
$$
\begin{array}{c}
\si(1) <  \si(2) < \dots < \si(p_1),  \\[0.3cm]
\si(p_1+1) <  \si(p_1+2) < \dots < \si(p_1+p_2), \\[0.3cm]
\dots   \\[0.3cm]
\si(n-p_k+1) <  \si(n-p_k+2) < \dots < \si(n)\,.
\end{array}
$$

We tacitly assume the Koszul sign rule. In particular,   
$$
(-1)^{\ve(\si; v_1, \dots, v_m)}
$$ 
will always denote the sign factor corresponding to the permutation $\si \in S_m$ of 
homogeneous vectors $v_1, v_2, \dots, v_m$. Namely, 
\begin{equation}
\label{ve-si-vvv}
(-1)^{\ve(\si; v_1, \dots, v_m)} := \prod_{(i < j)} (-1)^{|v_i | |v_j|}\,,
\end{equation}
where the product is taken over all inversions $(i < j)$ of $\si \in S_m$.

For a finite group $G$ acting on a cochain complex
(or a graded vector space) $V$, we denote by 
$$
V^G  \qquad \textrm{ and } \qquad  V_G,
$$
respectively, the subcomplex of $G$-invariants in $V$ and
the quotient complex of $G$-coinvariants. Using the advantage of 
the zero characteristic, we often identify $V_G$ with $V^G$ 
via this isomorphism
\begin{equation}
\label{coinvar-invar}
v ~\mapsto~ \sum_{g \in G} g (v) \maps V_G \to V^G\,. 
\end{equation}

For a graded vector space (or a cochain complex) $V$,
the notation $S(V)$ (resp. $\und{S}(V)$) is reserved for the
underlying vector space of the symmetric algebra (resp. the truncated symmetric algebra) of $V$: 
$$
S(V) = \bbk \oplus V \oplus S^2(V) \oplus S^3(V) \oplus \dots\,, 
$$ 
$$
\und{S}(V) =  V \oplus S^2(V) \oplus S^3(V) \oplus \dots\,,
$$
where 
$$
S^n(V) = \big( V^{\otimes_{\bbk}\, n} \big)_{S_n}\,.
$$

We denote by $\Com$ (resp. $\Lie$) the operad governing 
commutative (and associative) algebras without unit 
(resp.  the operad governing Lie algebras).
Furthermore, we denote by $\coCom$ 
the cooperad which is obtained from $\Com$ by 
taking the linear dual. The coalgebras over $\coCom$ are 
cocommutative (and coassociative) coalgebras without counit. 

The notation $\Cobar$ is reserved for the cobar construction \cite[Section 3.7]{notes}. 
 
For an operad (resp. a cooperad) $P$ and a cochain complex $V$ we denote by 
$P(V)$ the free $P$-algebra (resp. the cofree\footnote{In this paper we only consider 
nilpotent coalgebras.} $P$-coalgebra) generated by $V$: 
\begin{equation}
\label{P-Schur-V}
P(V) : = \bigoplus_{n \ge 0} \Big( P(n) \otimes V^{\otimes \, n} \Big)_{S_n}\,.
\end{equation}
For example,
$$
\Com(V) =  \coCom(V) = \und{S}(V)\, 
$$
as vector spaces.
We denote by $\mS$ the underlying collection of the 
endomorphism operad 
$$
\End_{\bsi \bbk}
$$
of the 1-dimensional space $\bsi \bbk$ placed in degree $-1$. 
The $n$-the space of $\mS$ is 
$$
\mS(n) = \sgn_n \otimes \bs^{n-1} \bbk\,,
$$
where $\sgn_n$ denotes the sign representation of the symmetric group $S_n$\,.
Recall that $\mS$ is naturally an operad and a cooperad.  

For a (co)operad $P$, we denote 
by $\mS P$ the (co)operad which is obtained from 
$P$ by tensoring with $\mS$: 
$$
\mS P : = \mS \otimes P\,.   
$$
It is clear that tensoring with 
$$
\mS^{-1} : = \End_{\bs \bbk}
$$
gives us the inverse of the operation $P \mapsto \mS P$\,.

For example, the dg operad $\Lie_{\infty} : = \Cobar(\mS^{-1}\coCom)$ governs 
$L_{\infty}$-algebras and the dg operad
\begin{equation}
\label{sLie-oper}
\sLie = \Cobar(\coCom)
\end{equation}
governs $\sLie$-algebras.

In this paper, we often call $\sLie$-algebras {\it shifted $L_{\infty}$-algebras}. 
Although a $\sLie$-algebra structure on a cochain complex $V$ is the same thing 
as an $L_{\infty}$ structure on $\bs V$\,, working with $\sLie$-algebras has 
important technical advantages. This is why we prefer to deal with shifted 
$L_{\infty}$-algebras (a.k.a.  $\sLie$-algebras) from the outset. 

The abbreviation ``MC'' is reserved for the term ``Maurer-Cartan''.

\section{Filtered $\sLie$ algebras, Maurer-Cartan (MC) elements, and twisting}
\label{filtered_sec}

Let $V$ be a cochain complex and $\cC$ be a coaugmented dg cooperad. 
We recall\footnote{See, for example, \cite[Corollary 5.3]{notes}.} that $\Cobar(\cC)$-algebra
structures on $V$ are in bijection with coderivations $Q$ of the $\cC$-coalgebra 
$\cC(V)$ satisfying the condition 
$$
Q \Big|_{V} = 0
$$ 
and the MC equation
$$
[\pa, Q] + \frac{1}{2} [Q,Q] = 0\,, 
$$
where $\pa$ comes from the differential on $V$ and $\cC$ (if $\cC$ has a 
non-zero differential).
 
Thus, since $\sLie$-algebras are $\Cobar(\coCom)$-algebras, a $\sLie$-structure 
on a cochain complex $(L, \pa)$ is a degree $1$ coderivation $Q$ of the 
cofree cocommutative coalgebra 
\begin{equation}
\label{coCom-L}
\coCom(L) = \und{S}(L) 
\end{equation}
satisfying the MC equation 
\begin{equation}
\label{MC-Q}
[\pa, Q] + \frac{1}{2} [Q,Q] = 0 
\end{equation}
and the additional condition
\begin{equation}
\label{Q-on-L}
Q \Big|_{L} = 0\,. 
\end{equation}

Introducing a degree $1$ coderivation $Q$ on \eqref{coCom-L} 
satisfying \eqref{Q-on-L} 
is equivalent to introducing the infinite collection of degree $1$ operations  
\begin{equation}
\label{mult-brack}
\{\underbrace{\cdot, \cdot, \dots, \cdot}_{m \textrm{ inputs}}\}  : S^m(L) \to L  
\end{equation}
for all $m \ge 2$. Namely, 
\begin{equation}
\label{mult-brack-Q}
\{v_1, v_2, \dots, v_m\} : = p_{L} \circ Q(v_1 v_2 \dots v_m)\,,
\end{equation}
where $p_L$ is the canonical projection $p_L : \und{S}(L) \to L$\,.

MC equation \eqref{MC-Q} is equivalent to 
the following sequence of relations on the brackets $\{\cdot,\cdot, \dots, \cdot\}$: 
\begin{multline}
\label{sLie-relations}
\pa \{v_1, v_2, \dots, v_m\} + 
\sum_{i=1}^m (-1)^{|v_1| + \dots + |v_{i-1}|} 
\{v_1,  \dots, v_{i-1},  \pa v_i, v_{i+1}, \dots, v_{m}\} \\
+ \sum_{k=2}^{m-1} 
\sum_{\si \in \Sh_{k, m-k} } 
(-1)^{\ve(\si; v_1, \dots, v_m)}
\{ \{v_{\si(1)}, \dots, v_{\si(k)} \}, v_{\si(k+1)}, \dots, v_{\si(m)} \} = 0\,,
\end{multline}
where $(-1)^{\ve(\si; v_1, \dots, v_m)}$ is the Koszul sign factor (see eq.  \eqref{ve-si-vvv}).

An $\infty$-morphism $F$ from a $\sLie$-algebra $(L, \pa, Q)$ to a $\sLie$-algebra 
$(\ti{L}, \ti{\pa}, \ti{Q})$ is a homomorphism 
$$
F : ( \und{S}(L), \pa + Q) \to (\und{S}(\ti{L}), \ti{\pa} + \ti{Q})   
$$
of the corresponding dg cocommutative coalgebras. 
Recall that any such coalgebra homomorphism $F$ is uniquely determined 
by its composition $F'$ with the projection $p_{\ti{L}}: \und{S}(\ti{L}) \to \ti{L}$: 
$$
F' : = p_{\ti{L}} \circ F : \und{S}(L) \to \ti{L} 
$$  
More precisely, given a degree zero map $F'  : \und{S}(L) \to \ti{L}$, the homomorphism 
of coalgebras $F$ is restored by the formula  
\begin{multline}
\label{F-F-pr}
F(v_1 v_2 \dots v_n) = 
 \sum_{t \ge 1}
\sum_{\substack{ k_1+ \dots + k_t = n \\[0.1cm] k_j \ge 1 }} 
\sum_{\si \in \mSH_{k_1, k_2, \dots, k_t}}  
F'(v_{\si(1)} \dots v_{\si(k_1)}) F'(v_{\si(k_1+1)} \dots v_{\si(k_1 + k_2)}) \dots \\
\dots F'(v_{\si(n-k_t+1)} \dots v_{\si(n)})\,,
\end{multline}
where $\mSH_{k_1, k_2, \dots, k_t}$ 
is  the subset of permutations in $\Sh_{k_1, k_2, \dots, k_t}$ satisfying 
the condition 
$$
\si(1) < \si(k_1+1) <  \si(k_1+k_2+1) < \dots < \si(n-k_t+1)\,.
$$
The compatibility of $F$ with the differentials $\pa + Q$ and $\pa + \ti{Q}$
is equivalent to the sequence of equations 
\begin{multline}
\label{F-and-diff-Q}
\ti{\pa} F' (v_1, v_2, \dots, v_m) - \sum_{i=1}^m (-1)^{|v_1| + \dots + |v_{i-1}|} F'(v_1, \dots, v_{i-1}, \pa v_i, v_{i+1}, \dots v_m) \\
- \sum_{p=2}^{m-1} \sum_{\si \in \Sh_{p, m-p}}  (-1)^{\ve(\si; v_1, \dots, v_m)}
F' (\{ v_{\si(1)}, \dots, v_{\si(p)} \}, v_{\si(p+1)}, \dots,  v_{\si(m)}) =  \\
\sum_{\substack{p_1 + p_2 + \dots + p_k = m \\ p_j \ge 1}}^{k \ge 2} \,
\sum_{\tau \in \Sh_{p_1, p_2, \dots, p_k} }
(-1)^{\ve(\tau; v_1, \dots, v_m)}  \{ F'(v_{\tau(1)}, \dots, v_{\tau(p_1)}), F'(v_{\tau(p_1+1)}, \dots, v_{\tau(p_1+p_2)}), 
\dots \\ \dots, F' (v_{\tau(m-p_k+1)}, \dots, v_{\tau(m)}) \}\ti{~}\,.
\end{multline}

\begin{defi}
\label{dfn:sLie-filtered}
We say that a $\sLie$ algebra $(L, \pa, \{\cdot,\cdot\},
\{\cdot,\cdot,\cdot\},\ldots)$ is \emph{filtered} if the underlying
complex $(L, \pa)$ is equipped with a complete descending filtration,
\begin{equation}
\label{filtr-L}
L = \cF_{1}L \supset \cF_{2}L \supset  \cF_{3}L  \cdots
\end{equation}
\begin{equation}
\label{L-complete}
L =\varprojlim_{k} L/\cF_{k}L\,,
\end{equation}
which is compatible with the brackets, i.e.
\[
\Bigl \{\cF_{i_{1}}L,\cF_{i_{2}}L,\ldots,\cF_{i_{m}}L \Bigr \} \subseteq
\cF_{i_{1} + i_{2} + \cdots + i_{m}} L \quad \forall ~~ m >1.
\]
\end{defi}

The filtration \eqref{filtr-L} induces a natural descending 
filtration and hence a topology on $\und{S}(L)$. 
{\bf In this paper,    
we tacitly assume that $\infty$-morphisms of filtered $\sLie$-algebras 
are continuous with respect to this topology.} 

Conditions \eqref{L-complete} and $L = \cF_1 L$ imply that 
the $\sLie$-algebra $L$ is pronilpotent, and hence the left hand side of
the \textit{MC equation}:
\begin{equation} 
\label{MC_equation}
\pa \alpha + \sum_{m=2}^{\infty} \frac{1}{m!}
\bigl \{\underset{m}{\underbrace{\alpha,\alpha,\ldots,\alpha}} \bigr \} = 0
\end{equation}
is well-defined for every $\alpha \in L$. 

A \und{degree zero} element $\al \in L$ satisfying equation \eqref{MC_equation} 
is called a \textit{MC element}  of $L$.  In this paper, the notation 
$$
\MC(L)
$$
is reserved for the set of MC elements of a filtered $\sLie$-algebra $L$.

We next record some well known\footnote{For more details about twisting we refer the 
reader to \cite[Section 2.4]{thesis}, \cite{DeligneTw}, \cite[Section 4]{Ezra-infty}}
facts regarding MC elements:
\begin{enumerate}
\item{
Given a filtered $\sLie$ algebra $L=(L, \pa, \{\cdot,\cdot\}\{\cdot,\cdot,\cdot \},\ldots)$
and a MC element $\alpha \in L$, we can construct a new filtered
$\sLie$ structure $L^{\alpha} = (L , \pa^{\al}, \{\cdot,\cdot\}^{\alpha},\{\cdot,\cdot,\cdot\}^{\alpha},\ldots)$ 
on the cochain complex $(L, \pa^{\al})$ with the new differential 
\begin{equation}
\label{diff-twisted}
\pa^{\al} (v) : = \pa (v)  + \sum_{k=1}^{\infty} \frac{1}{k!} \{\underbrace{\al, \dots, \al,}_{k \textrm{ times}} v\}  
\end{equation}
and the new multi-brackets
\begin{equation}
\label{Linft_twisted}
\{v_1,v_2,\cdots, v_m \}^{\alpha}:=
\sum_{k=0}^{\infty} \frac{1}{k!} \{\underbrace{\al, \dots, \al,}_{k \textrm{ times}}
v_1, v_2,\cdots, v_m \}
\end{equation}
}

\item{ If $F \maps L \to \ti{L}$ is an $\infty$-morphism of $\sLie$ algebras and
$\alpha \in L$ is a MC element then
\begin{equation} 
\label{F_MC_elt_eq}
F_*(\alpha): = \sum_{k=1}^{\infty} \frac{1}{k!} p_{\ti{L}} \circ F \bigl ( \al^k \bigr )
\end{equation}
is a MC element of $\ti{L}$ \footnote{Note that $F_*(\alpha)$
  is well defined for any degree 0 element $\alpha$.}.
}

\item{
If $F \maps L \to \ti{L}$ is an $\infty$-morphism of $\sLie$ algebras and
$\alpha \in L$ is a MC element then we can construct a new 
$\infty$-morphism
\[
F^{\alpha} \maps L^{\alpha} \to \ti{L} ^{F_*(\alpha)}
\]
with 
\begin{equation} 
\label{twisted_map}
p_{\ti{L}} \circ F^{\alpha} (v_1 v_2 \dots v_m) = \sum_{k=0}^{\infty} \frac{1}{k!}\, 
p_{\ti{L}} \circ F (\al^k \, v_1 v_2 \dots v_m)\,.
\end{equation}
}

\end{enumerate}
Note that the infinite sums in equations \eqref{F_MC_elt_eq} and \eqref{twisted_map} are well defined because $F$ is continuous and $L = \cF_1 L$\,.

More generally, we denote by $\curv$ the map of sets $L^0 \mapsto L^1$ 
given by the formula 
\begin{equation}
\label{def-curv}
\curv(\al) : = \pa \al + \sum_{m \ge 1} \frac{1}{m!} \{\al, \dots, \al\}_m\,. 
\end{equation}
For example, elements $\al \in L^0$ satisfying $\curv(\al) = 0$ are precisely 
MC elements of the $\sLie$-algebra $L$.
Various useful properties of the operation $\curv$ are listed in the following proposition. 
\begin{prop}
\label{prop:curv}
Let $L$ and $\ti{L}$ be filtered $\sLie$-algebras and $F$
be a continuous $\infty$-morphism from $L$ to $\ti{L}$. 
Then for every $\al, \beta \in L^0,  v \in L$ we have 
\begin{equation}
\label{Bianchi}
\pa(\curv(\al)) + \sum_{m=1}^{\infty} \frac{1}{m!} \{\underbrace{\al,
  \dots, \al}_{\text{$m$ {\rm times}}}, \curv(\al)\}= 0\,,
\end{equation}
\begin{equation}
\label{U-star-curv}
\curv \big( F_*(\al) \big) = 
\sum_{m \ge 0} \frac{1}{m!} F' \big( \al^m  \curv(\al) \big)\,,
\end{equation}

\begin{equation}
\label{square-curv}
\pa^{\al} \circ \pa^{\al} (v) =  - \{\curv(\al), v\}^{\al}_2\,,
\end{equation}

\begin{equation}
\label{curv-sum}
\curv(\al + \beta) = \curv(\al) +  \pa^{\al}(\beta) +
\sum_{m=2}^{\infty} \frac{1}{m!} \{\underbrace{\beta, \dots,
  \beta}_{\text{$m$ {\rm  times}}} \}^{\al}\,.
\end{equation}
\end{prop}
\begin{proof}
Equation \eqref{Bianchi} is proved in \cite[Lemma 4.5]{Ezra-infty}. 
Equations \eqref{square-curv} and \eqref{curv-sum} follow from relations \eqref{sLie-relations}.

To prove \eqref{U-star-curv} we use equation \eqref{F-F-pr} which implies that 
\begin{equation}
\label{U-cxp-al}
F \big( \exp(\al)-1 \big) = \exp\big( F_*(\al) \big) -1\,, 
\end{equation}
where $\exp(\al)-1$ (resp. $ \exp\big( F_*(\al) \big) -1$) is considered 
as the element of the completion of $\und{S}(L)$ (resp. $\und{S}(\ti{L})$)
defined by the corresponding Taylor series. 

A similar computation shows that, for every $\beta \in \ti{L}$, we have 
\begin{equation}
\label{ti-Q-cxp-beta}
\ti{Q} \big( \exp(\beta) -1 \big) = \exp(\beta) \curv(\beta)\,,
\end{equation}
where $\ti{Q}$ is the coderivation\footnote{Here we tacitly assume that 
$Q$ and $\ti{Q}$ are extended in the natural way to the completions of 
$\und{S}(L)$ and $\und{S}(\ti{L})$, respectively.} 
of $\und{S}(\ti{L})$ corresponding to the $\sLie$-algebra structure on $\ti{L}$. 

On the other hand, 
$$
\ti{Q} \circ F = F \circ Q\,.
$$

Hence \eqref{U-cxp-al} and \eqref{ti-Q-cxp-beta} imply that
\begin{equation}
\label{curv-done}
\ti{Q} \big( \exp \big( F_*(\al) \big) - 1 \big) =
F \big( \exp(\al) \curv(\al) \big)\,.
\end{equation}

Applying the canonical projection $p_{\ti{L}}$ to both sides of 
\eqref{curv-done} we get desired identity \eqref{U-star-curv}. 
\end{proof}

\begin{remark}
\label{rem:diff-brack1}
It is often convenient to ``absorb'' the differential $\pa$ on a $\sLie$-algebra $L$ into the 
collection of multi-brackets \eqref{mult-brack} treating it as the 
unary operation: 
\begin{equation}
\label{brack-1}
\{ \cdot \} : = \pa : L \to L\,.
\end{equation}
Then relations \eqref{sLie-relations} can be rewritten in the more concise form
\begin{equation}
\label{sLie-rel-simpler}
\sum_{k=1}^{m} 
\sum_{\si \in \Sh_{k, m-k} } 
(-1)^{\ve(\si; v_1, \dots, v_m)}
\{ \{v_{\si(1)}, \dots, v_{\si(k)} \}, v_{\si(k+1)}, \dots, v_{\si(m)} \} =0\,,
\end{equation}
where $(-1)^{\ve(\si; v_1, \dots, v_m)}$ is, as above, the sign coming
from the Koszul rule. 

Thus one can say that a $\sLie$-structure on a cochain complex 
$(L, \pa)$ is a collection of degree $1$ multi-brackets
\begin{equation}
\label{mult-brack-more}
\{\underbrace{\cdot, \cdot, \dots, \cdot}_{m \textrm{ inputs}}\}  : S^m(L) \to L  
\end{equation}
for $m \ge 1$ satisfying \eqref{sLie-rel-simpler} and the condition 
\begin{equation}
\label{brack-1-diff}
\{v \}  = \pa \, v\,, \qquad \forall ~~ v \in L\,.
\end{equation}
\end{remark}

\section{The symmetric monoidal category $\tLie$} 

In this section we introduce the main hero of this note: an ``enhanced'' version $\tLie$ of the category of 
$\sLie$-algebras. Then, we will show that $\tLie$ is a symmetric monoidal category.      

Objects of the category $\tLie$ are filtered $\sLie$ algebras and 
morphisms are defined in the following way:
\begin{defi}
\label{dfn:enhanced}
An {\bf enhanced morphism} 
\[
L_1 \xto{(\al, F)} L_2
\]
between filtered $\sLie$ algebras is a pair consisting of a MC element
$\alpha \in L_2$ and a continuous $\infty$-morphism
$F \maps L_1 \to L^{\alpha}_2$.
\end{defi}
Note that every $\infty$-morphism $F$
of filtered $\sLie$ algebras is canonically the
enhanced morphism $(0,F)$.

Let $(\al, F)$ be an enhanced morphism from $L_1$ to $L_2$. 
Then $F$ is a homomorphism of dg $\coCom$-coalgebras 
\begin{equation}
\label{F-coalg-morph}
F : (\und{S}(L_1), Q_1) ~\to~ (\und{S}(L_2), Q^{\al}_2)\,,
\end{equation}
where $Q_1$ and $Q^{\al}_2$ are the codifferentials defining 
the $\sLie$-structures on $L_1$ and $L^{\al}_2$, respectively. 

To define the composition of enhanced morphisms, we 
need to extend the codifferentials $Q_1$ and $Q_2$ to
the completions 
\begin{equation}
\label{comp-full-symm}
S(L_1)\,\hat{} ~~\textrm{ and }~~ S(L_2)\,\hat{}
\end{equation}
with respect to the natural descending filtrations coming 
from those on $L_1$ and $L_2$, respectively. We do this
via extending $Q_1$ and $Q_2$ by continuity and setting
\begin{equation}
\label{ext-ing-Q12}
Q_1 (1) = Q_2(1) = 0\,.  
\end{equation}

\begin{remark}
\label{rem:coalg-comp}
Note that the completion $S(L)\, \hat{}$ of $S(L)$ with respect 
to the filtration coming from some $L$ is, strictly speaking, not 
a coalgebra because the natural 
analog of the comultiplication $\D$ on $S(L)\, \hat{}$ 
lands in the completed tensor product
$$
S(L)\, \hat{} ~ \wh{\otimes} ~ S(L)\, \hat{}\,.
$$
This subtlety should be kept in mind when writing 
equations like
\begin{equation}
\label{Delta-Q}
\D Q (X) = (Q \otimes 1 + 1 \otimes Q) \D(X) 
\end{equation}
or claiming that $Q$ is a coderivation of the ``coalgebra''
$S(L)\, \hat{}$.  
\end{remark}

Passing to the completions allows us to consider such 
elements as $e^{\al}$, where $\al \in L_2^0$, 
and  rewrite the MC equation \eqref{MC_equation} for $\al$ 
in the form 
\begin{equation}
\label{MC-eq-easy}
Q (e^{\al}) = 0\,.
\end{equation}

Let us also extend the coalgebra morphism \eqref{F-coalg-morph} to the morphism 
\begin{equation}
\label{F-coalg-morph-hat}
F : (S(L_1)\,\hat{}, Q_1) ~\to~ (S(L_2)\,\hat{}, Q^{\al}_2)
\end{equation}
by continuity and declaring that 
\begin{equation}
\label{F-at-1}
F(1) : = 1\,. 
\end{equation}
It is easy to see that this extension is compatible with the comultiplications 
on $S(L_1)\,\hat{}$ and $S(L_2)\,\hat{}$, respectively:
\begin{equation}
\label{new-F-Delta}
\Delta \circ F = (F \otimes F) \circ \Delta\,.
\end{equation}

Since the codifferential $Q_2^{\al}$ can be written in the form 
\begin{equation}
\label{Q-2-twisted}
Q_2^{\al} = e^{-\al} Q_2 e^{\al}
\end{equation}
the compatibility of $F$ with $Q_1$ and $Q_2^{\al}$ is equivalent to the equation  
\begin{equation}
\label{F-Q1-Q2-al}
Q_2 (e^{\al} F) = (e^{\al} F) Q_1\,.
\end{equation}

On the other hand, $e^{\al} F$ is obviously compatible with
comultiplications on  $S(L_1)\,\hat{}$ and $S(L_2)\,\hat{}$. 
Thus we arrive at the following statement:
\begin{claim}
\label{cl:enhanced}
Every enhanced morphism $(\al, F)$ from $L_1$ to $L_2$
gives rise to the following homomorphism of cocommutative
``coalgebras''
$$
U_{\al, F} : (S(L_1)\,\hat{}, Q_1) \to (S(L_2)\,\hat{}, Q_2)
$$
\begin{equation}
\label{exp-al-F}
U_{\al, F} (X) : = e^{\al} F(X)\,,
\end{equation}
where 
\begin{equation}
\label{F-of-1}
F(1) : = 1\,,
\end{equation}
and $F$ is extended to the completion $S(L_1)\,\hat{}$ by continuity. \qed
\end{claim}

Let $L_1, L_2, L_3$ be filtered $\sLie$-algebras and let 
$(\al_2, F)$ (resp. $(\al_3, G)$) be an enhanced morphism from 
$L_1$ to $L_2$ (resp. from $L_2$ to $L_3$). According to the above 
claim, $(\al_2, F)$ and $(\al_3, G)$ give us morphisms of cocommutative 
``coalgebras'':    
\begin{equation}
\label{exp-al2-F}
U_{\al_2, F} : (S(L_1)\,\hat{}, Q_1) \to (S(L_2)\,\hat{}, Q_2)
\end{equation}
and 
\begin{equation}
\label{exp-al3-G}
U_{\al_3, G} : (S(L_2)\,\hat{}, Q_2) \to (S(L_3)\,\hat{}, Q_3)\,.
\end{equation}

We observe that the composition $U_{\al_3, G} \circ U_{\al_2, F}$ acts 
on $X \in S(L_1)\,\hat{}$ as 
\begin{multline}
\label{composition-calc}
U_{\al_3, G} \circ U_{\al_2, F} (X) = e^{\al_3} G \big(e^{\al_2} F(X)\big) = \\ 
e^{\al_3} e^{G_*(\al_2)} \circ (e^{- G_*(\al_2)} G e^{\al_2}) \circ F (X)= 
e^{\al_3 + G_*(\al_2)} G^{\al_2} \circ F (X)\,.
\end{multline}

Therefore the composition  $U_{\al_3, G} \circ U_{\al_2, F}$ is a morphism 
from $S(L_1)\,\hat{}$ to $S(L_3)\,\hat{}$ of the form \eqref{exp-al-F} corresponding 
to the pair 
\begin{equation}
\label{composition}
(\al_3 + G_*(\al_2), G^{\al_2} \circ F)\,. 
\end{equation}

Thus we proved the following statement: 
\begin{prop} 
\label{comp-prop-Linfty-enh}
Equation \eqref{composition} defines a composition 
of enhanced morphisms for $\sLie$-algebras. Moreover, 
this composition is associative.  \qed
\end{prop}

\subsection{$\tLie$ is a symmetric monoidal category} 

Given two filtered $\sLie$ algebras $(L, \{\cdot,\cdot\},
\{\cdot,\cdot, \cdot\},\ldots)$ and $(\wt{L}, \{\cdot,\cdot \}\wt{~},
\{\cdot,\cdot, \cdot\}\wt{~},\ldots)$, one obtains a
filtered $\sLie$ structure on the direct sum $L \oplus \wt{L}$ 
by setting
\[
\{x_1 + x'_1, x_2 + x'_2,\ldots, x_k + x'_k\} :=
\{x_1, x_2,\ldots, x_k\} + \{x'_1, x'_2,\ldots, x'_k\}\wt{~}\,,
\] 
and 
$$
\cF_k  (L \oplus \wt{L}) : = (\cF_k L) \oplus  (\cF_k \wt{L})\,.
$$

If $\al$ and $\wt{\al}$ are MC elements of $L$ and $\wt{L}$, respectively, 
then $\al + \wt{\al} \in L \oplus \wt{L}$ is clearly a MC element of 
the $\sLie$-algebra $L\oplus \wt{L}$. Furthermore, the operation of 
twisting (by a MC element) is compatible with $\oplus$, i.e. the $\sLie$-algebra
$L^{\al} \oplus \wt{L}^{\wt{\al}}$ is canonically isomorphic to the $\sLie$-algebra
to $(L\oplus \wt{L})^{\al + \wt{\al}}$\,.

Let us now consider a pair of enhanced morphisms 
\begin{equation}
\label{two-enh-morph}
(\al, F) : L_1 \to L_2\,,  \qquad 
\textrm{and} \qquad 
(\wt{\al}, \wt{F}) : \wt{L}_1 \to \wt{L}_2\,.
\end{equation}

Extending the corresponding $\infty$-morphisms 
$$
F : \und{S}(L_1) \to  \und{S}(L^{\al}_2)\,, 
\qquad 
\wt{F} : \und{S}(\wt{L}_1) \to  \und{S}(\wt{L}^{\wt{\al}}_2)\,,  
$$
to $S(L_1)$ and $S(\wt{L}_1)$ respectively by declaring that 
$$
F(1) : = 1\,,  \qquad 
\textrm{and} \qquad \wt{F}(1) : = 1 
$$ 
and tensoring the resulting homomorphism of
cocommutative coalgebras (with counits) we get 
\begin{equation}
\label{F-otimes-wtF}
F \otimes \wt{F} : S(L_1) \otimes S(\wt{L}_1) \cong 
S(L_1 \oplus \wt{L}_1) \to  S(L^{\al}_2) \otimes S(\wt{L}^{\wt{\al}}_2) \cong S(L^{\al}_2 \oplus \wt{L}^{\wt{\al}}_2)\,. 
\end{equation}
 
Since the operation of twisting by a MC element commutes with $\oplus$, the restriction 
of \eqref{F-otimes-wtF} to $\und{S}(L_1 \oplus \wt{L}_1)$
gives us a desired $\infty$-morphism 
\begin{equation}
\label{F-oplus-wtF}
F \otimes \wt{F} \Big|_{\und{S}(L_1 \oplus \wt{L}_1)}   ~:~ \und{S}(L_1 \oplus \wt{L}_1) ~ \to ~ 
\und{S}\big( (L_2 \oplus \wt{L}_2)^{\al + \wt{\al}} \big)
\end{equation}
from the $\sLie$-algebra $L_1 \oplus \wt{L}_1$ to the $\sLie$-algebra 
$(L_2 \oplus \wt{L}_2)^{\al + \wt{\al}}$\,. 

Thus every pair \eqref{two-enh-morph} gives us 
the enhanced morphism 
\begin{equation}
\label{sum-of-morph}
(\al, F) \otimes (\wt{\al}, \wt{F}) : = \Big( \al + \wt{\al}, F \otimes \wt{F}  \Big|_{\und{S}(L_1 \oplus \wt{L}_1)}  \Big)
\end{equation}
from $L_1 \oplus \wt{L}_1$ to $L_2 \oplus \wt{L}_2$\,.

It is easy to verify that the operations $(L,\wt{L}) \mapsto L \oplus \wt{L}$
and \eqref{sum-of-morph} give
$\tLie$ the structure of a symmetric monoidal
category whose unit object is $\bfzero$.

\section{Integration of a $\tLie$-enriched category to a simplicial category}

In this section, we show that every $\tLie$-enriched category $\mC$ can be integrated to 
a simplicial category. We use this construction in subsequent paper \cite{HAform} to find 
a higher categorical structure formed by homotopy algebras 
of a fixed type. 

For this purpose, we need to recall the construction \cite{Ezra-infty} of the 
Deligne-Getzler-Hinich (DGH) $\infty$-groupoid $\mMC_{\bul}(L)$ of a nilpotent 
$\sLie$-algebra $L$\,.

Let $\Omega_{n}=\Omega^{\bullet}(\Delta^{n})$ denote the polynomial de
Rham complex on the $n$-simplex with coefficients in $\bbk$, and
$\{\Omega_{n} \}_{n \geq 0}$ the associated simplicial
dg commutative $\bbk$-algebra. 

Since, for every $n$, $\Om_n$ is a dg commutative algebra, the tensor product 
$$
L \otimes \Om_n
$$
is naturally a nilpotent $\sLie$-algebra. Furthermore, the simplicial 
structure on the collection $\{\Omega_{n} \}_{n \geq 0}$ gives us 
the structure of a simplicial set on the collection
\begin{equation}
\label{mMC-n-L}
\mMC_n(L) : = \MC(L \otimes \Om_{n})\,,
\end{equation}
where, as above, $\MC(\cL)$ denotes the set of MC elements of a $\sLie$-algebra $\cL$\,.

Due to \cite[Lemma 4.6]{Ezra-infty}, the simplicial set $\mMC_{\bul}(L)$ \eqref{mMC-n-L} is 
a Kan complex (a.k.a. an $\infty$-groupoid). 

In the examples we keep in mind \cite{thesis}, \cite{stable1}, 
\cite{HAform}, \cite{DP}, \cite{notes}, \cite{Chern}, \cite{DTT},
 \cite{DeligneTw}, \cite{Thomas-grt}, \cite{Thomas-KS} the $\sLie$-algebras are 
rarely nilpotent. However, the above construction can be easily extended to 
the case when the $\sLie$-algebra $L$ is filtered 
in the sense of Definition \ref{dfn:sLie-filtered}. 

In this case we have to replace $L \otimes \Om_n$ by the completed tensor product 
$$
L \hotimes \Om_n
$$
of the topological space $L$ (with the topology coming from the filtration) and 
the discrete topological space $\Om_n$. 

We claim that 
\begin{prop}
\label{prop:MC-is-Kan}
For every filtered $\sLie$-algebra $L$ the simplicial set with
\begin{equation}
\label{mMC-n-L-filtered}
\mMC_n (L) : =  \MC(L \hotimes \Om_{n})
\end{equation}
is a Kan complex.
\end{prop}
\begin{proof} 
The canonical maps $L/\cF_{k}L \to
L/\cF_{k-1}L$ are surjective strict morphisms between nilpotent
$\sLie$-algebras. Hence, Proposition 4.7 in \cite{Ezra-infty} implies that
the induced maps $\mMC_{\bullet}(L/\cF_{k}L) \to
\mMC_{\bullet}(L/\cF_{k-1}L)$ are fibrations between Kan complexes. 
Therefore, the inverse limit of this tower of fibrations is a Kan complex (cf.\
\cite{Goerss-Jardine}[Sec.\ VI.1]).
Our proposition then follows since 
$$
\mMC_{\bul} (L) = \varprojlim_k \mMC_{\bul} (L / \cF_{k}L)\,.
$$
\end{proof}

Let us next observe that any continuous $\infty$-morphism $F \maps L \to \tilde{L}$
of filtered $\sLie$-algebras 
gives a collection of $\infty$-morphisms of $\sLie$-algebras
\begin{equation}
\label{U-level-n}
\begin{split}
F^{(n)} \maps  L \hotimes \Om_n &~ \to ~ \ti{L} \hotimes \Om_n \\
F^{(n)} \bigl(v_1 \tensor \omega_1, v_2 \tensor \omega_2,\ldots,v_m \tensor
\omega_m \bigr) &= \pm F(v_1,v_2,\ldots,v_m)\tensor \omega_1 \omega_2 \cdots \omega_m,
\end{split}
\end{equation}
where $v_i \in L$, $\omega_i \in \Omega_n$, and $\pm$ is the usual
Koszul sign.
This collection is obviously compatible with 
all the faces and all the degeneracies.  Hence, $F$ induces a morphism of simplicial sets 
\begin{equation}
\label{MC-U}
\mMC_{\bul}(F) : \mMC_{\bul}(L) \to  \mMC_{\bul}(\ti{L})
\end{equation}
given by the formula 
\begin{equation}
\label{MC-U-formula}
\mMC_{n}(F) (\al) : = F^{(n)}_*(\al)\,.
\end{equation}

\begin{prop} \label{prop:functor_yes}
The assignment 
$$
L \mapsto \mMC_{\bul}(L)
$$
extends naturally to a monoidal functor from the 
category $\tLie$ to the category of simplicial sets.
\end{prop}
For the proof of Prop.\ \ref{prop:functor_yes}, we need the following lemma.

\begin{lem}
\label{lem:shift}
Let $\al$ be a MC element in $L$ and $L^{\al}$ be the filtered $\sLie$ algebra 
which is obtained from $L$ via twisting by $\al$. Then the following assignment 
\begin{equation}
\label{eq:zero-base}
\beta \in \MC\big( L^{\al} \hotimes \Om_n \big)  ~~\mapsto ~~
\al + \beta \in \MC\big( L \hotimes \Om_n \big)
\end{equation}
is an isomorphism of simplicial sets 
\begin{equation}
\label{Shift-al}
\Shift_{\al} : \mMC_{\bul}(L^{\al})  \to  \mMC_{\bul}(L)
\end{equation}
which sends the zero MC element of $L^{\al}$ to the MC 
element $\al$ in $L$. For every $\infty$-morphism $F$ of filtered 
$\sLie$-algebras $L \to \ti{L}$ the following diagram commutes: 
\begin{equation}
\label{diag-Shift}
\begin{tikzpicture}
\matrix (m) [matrix of math nodes, row sep=2.5em, column sep=4em]
{    \mMC_{\bul}(L^{\al})  &  \mMC_{\bul}(L)  \\
  \mMC_{\bul}(\,\ti{L}^{F_*(\al)}\,)  &  \mMC_{\bul}(\,\ti{L}\,)\,, \\ };
\path[->,font=\scriptsize]
(m-1-1) edge node[auto] {$\Shift_{\al}$}  (m-1-2)  edge node[left] {$\mMC_{\bul}(F^{\al})$} (m-2-1)
(m-1-2)  edge node[auto] {$\mMC_{\bul}(F)$}  (m-2-2) 
(m-2-1) edge node[auto] {$\Shift_{F_*(\al)}$} (m-2-2);
\end{tikzpicture}
\end{equation}
where $F^{\al}$ denotes the $\infty$-morphism $L^{\al} \to \ti{L}^{F_*(\al)}$  which is obtained 
from $F$ via twisting by the MC element $\al$. 
\end{lem}
\begin{proof} 
Eq.\ \eqref{curv-sum} from Prop.\ \ref{prop:curv} implies that map \eqref{eq:zero-base} is well defined, and it is clearly
injective. If $\gamma \in \MC\big( L \hotimes \Om_n \big)$, then
$\curv(\gamma) = \curv\bigl(\alpha + (\gamma -\alpha)
\bigr)=0$. Hence, \eqref{curv-sum} implies that
\[
\gamma -\alpha \in \MC\big( L^{\al} \hotimes \Om_n \big), 
\] 
so map \eqref{eq:zero-base} is also surjective. Since $\alpha$ is
constant as an element of $L \hotimes \Omega_\bullet$,
\eqref{eq:zero-base} induces the isomorphism of simplicial sets
\eqref{Shift-al}.

To show that diagram \eqref{diag-Shift} commutes,
suppose that $\beta \in \mMC_{\bul}(L^{\al})$. Then we have
\begin{equation} \label{eq:shift-sum1}
\begin{split}
\bigl(\mMC_{\bul}(F) \circ \Shift_{\al} \bigr) (\beta) &= F_{\ast}(\al + \beta) 
= \sum_{k \geq 1} \frac{1}{k!} F'\bigl( (\alpha + \beta)^k \bigr) 
=\sum_{k \geq 1} \sum_{l=0}^{k} \frac{1}{l! (k-l)!} F' \bigl(\alpha^l
\beta^{k-l} \bigr).
\end{split}
\end{equation}
On the other hand,
\begin{equation}
\begin{split}
\bigl(\Shift_{F_{\ast}(\alpha)} \circ \mMC_{\bul}(F^{\al}) \bigr)(\beta) &=
F_{\ast}(\alpha) + F^{\al}_{\ast}(\beta) = \sum_{l \geq 1} \frac{1}{l!} F'(\alpha^l) + \sum_{k \geq 1}
\frac{1}{k!}(F^{\alpha})' (\beta^k)\\
& =\sum_{l \geq 1} \frac{1}{l!} F'(\alpha^l) + \sum_{k \geq 1} \sum_{l
  \geq 0} \frac{1}{l!k!}F'(\alpha^l \beta^k).
\end{split}
\end{equation}
After rearranging terms, we see that the two compositions above are equal.
\end{proof}
\vspace{.5cm}

\begin{proof}[ of Prop. \ref{prop:functor_yes}]
To a morphism
\[
L_1 \xto{(\al, F)} L_2
\]
in $\tLie$, we assign the morphism of simplicial sets
\begin{equation}
\begin{split}
\mMC_{\bul}(\al,F) \maps \mMC_{\bul}(L_1) \to \mMC_{\bul}(L_2) \\ 
\mMC_{\bul}(\al,F) := \Shift_{\alpha} \circ F_{\ast}.
\end{split}
\end{equation}

Given another morphism in $\tLie$
\[
L_2 \xto{(\beta, G)} L_3,
\]
we need to show that this assignment respects composition, which,
because of Eq.\ \eqref{composition}, is equivalent to verifying the equality
\begin{equation} \label{eq:functor_yes_eq1}
\mMC_{\bul}(\beta + G_*(\al), G^{\al} \circ F) = \mMC_{\bul}(\beta,G)
\circ \mMC_{\bul}(\al,F).
\end{equation}
Observe that the $\sLie$ algebras
$(L^{\beta}_3)^{G_{\ast}(\al)}$ and $L_{3}^{\beta +G_{\ast}(\al)}$
are equal (not just isomorphic). Therefore, we have the following diagram 
\[
\begin{tikzpicture}
\matrix (m) [matrix of math nodes, row sep=2.5em, column sep=4em]
{    
\mMC_{\bul}(L_1)  &  \mMC_{\bul}(L^{\al}_2)  & \mMC_{\bul}(L_2) \\
& \mMC_{\bul} \bigl( (L_3^{\beta})^{G_{\ast}(\al)} \bigr) & \mMC_{\bul}(L^{\beta}_{3}) \\
&\mMC_{\bul}(L_{3}^{\beta +G_{\ast}(\al)}) & \mMC_{\bul}(L_3) \\
};
 \path[->,font=\scriptsize]
(m-1-1) edge node[auto] {$F_\ast$}  (m-1-2)  
(m-1-2) edge node[auto]{$\Shift_\al$} (m-1-3) 
(m-1-3) edge node[auto] {$G_\ast$} (m-2-3) 
(m-2-3) edge node[auto]{$\Shift_\beta$} (m-3-3)
(m-1-2) edge node[auto]{$G^{\al}_\ast$} (m-2-2)
(m-2-2) edge node[above]{$\Shift_{G_\ast(\al)}$} (m-2-3)
(m-2-2) edge node[above,sloped] {$=$}  (m-3-2)  
(m-3-2) edge node[above]{$\Shift_{\beta +G_{\ast}(\al)}$} (m-3-3)  
;
\end{tikzpicture}
\]
Lemma \ref{lem:shift} implies that the top rectangle of the diagram
commutes, and the definition of $\Shift_{\beta +G_{\ast}(\al)}$
implies that the lower rectangle commutes. Hence, Eq.\
\eqref{eq:functor_yes_eq1} holds.

Finally, given a pair of $\sLie$-algebras $L$, $\tilde{L}$, there is
a natural isomorphism
\[
\mMC_{\bul}(L \oplus \tilde{L}) \cong \mMC_{\bul}(L) \times \mMC_{\bul}(\tilde{L}).
\]
Indeed, for each $n \geq 0$, we have the following natural isomorphisms:
\begin{align*}
\varprojlim_k \bigl(L \oplus \tilde{L} /\cF_k(L  \oplus  \tilde{L} ) 
\otimes \Omega_n  \bigr) & \cong
\varprojlim_k \Bigl( \bigl(L/\cF_kL \oplus  \tilde{L}/\cF_k\tilde{L} \bigr)
\otimes \Omega_n \Bigr)  \\
& \cong  \varprojlim_k \bigl(L/\cF_kL \otimes \Omega_n \oplus  \tilde{L}/\cF_k\tilde{L} 
 \otimes \Omega_n \bigr) \\
& \cong \bigl( \varprojlim_k L/\cF_kL \otimes \Omega_n \bigr) \oplus 
\bigl ( \varprojlim_k \tilde{L}/\cF_k \tilde{L} \otimes \Omega_n \bigr),
\end{align*}
where the last line above follows from the fact that both the projective limit
and direct sum are limits, and hence commute. So we have exhibited a
natural isomorphism of $\sLie$-algebras
\[
(L \oplus \tilde{L}) \widehat{\otimes} \Omega_n \cong L
\widehat{\otimes} \Omega_n \oplus \tilde{L} \widehat{\otimes} \Omega_n.
\]
By combining this with the obvious natural isomorphism of sets:
\[
\MC \bigl (L \widehat{\otimes} \Omega_n \oplus \tilde{L}
\widehat{\otimes} \Omega_n \bigr) \cong 
\MC \bigl (L \widehat{\otimes} \Omega_n \bigr) \times \MC \bigl( \tilde{L}
\widehat{\otimes} \Omega_n \bigr),
\]
it follows that  $\mMC_{\bul}(-)$ is indeed a strong monoidal functor.
\end{proof}

\begin{remark}
As alluded to in the introduction, the functor constructed in
Prop.\ \ref{prop:functor_yes} demonstrates the utility of enhanced
morphisms. Simplicial morphisms which lie in the image of $\mMC_{\bul}(-)$ 
need not preserve the base point $0$. For example,
\[
\Hom_{\tLie}(\bfzero,L) \cong
\Hom_{\mathsf{sSet}} \bigl(\Delta^{0},\mMC_{\bul}(L) \bigr) \cong \MC(L).
\]
\end{remark}

Finally, let $\mC$ be a $\tLie$-enriched category \cite{Kelly}. Then, to every 
pair of objects $A, B$ of $\mC$, we may assign the Kan complex 
\begin{equation}
\label{Kan-A-B}
\mMC_{\bul}\big( \Map(A,B) \big)\,,
\end{equation}
where $\Map(A,B)$ denotes the mapping space (i.e. a filtered 
$\sLie$-algebra) corresponding to the pair $A,B$ in $\mC$.  
Since the functor described in Prop.\ \ref{prop:functor_yes}
is a monoidal functor from the 
category $\tLie$ to the category of simplicial sets, 
the Kan complexes \eqref{Kan-A-B} assemble into a category 
enriched over $\infty$-groupoids and we conclude that 
\begin{thm}
\label{thm:main}
For every $\tLie$-enriched category $\mC$ the assignment 
$$
(A,B) \in \Objects(\mC) \times  \Objects(\mC) ~ \mapsto ~ \mMC_{\bul}\big( \Map(A,B) \big)
$$
gives us a category enriched over $\infty$-groupoids 
(a.k.a. Kan complexes).  \qed
\end{thm}

~\\
~\\

\noindent\textsc{Department of Mathematics,\\
Temple University, \\
Wachman Hall Rm. 638\\
1805 N. Broad St.,\\
Philadelphia PA, 19122 USA \\
\emph{E-mail address:} {\bf vald@temple.edu} }

~\\

\noindent\textsc{Department of Mathematics\\
University of Louisiana at Lafayette\\
217 Maxim Doucet Hall, P.O. Box 43568 \\
Lafayette, LA 70504-3568 USA\\
\emph{E-mail address:} {\bf crogers@louisiana.edu}}

\end{document}